\documentclass[12pt]{amsart}
\usepackage{amssymb,bbold}
\usepackage[all]{xy}

\theoremstyle{plain}
\newtheorem{theorem}{Theorem}
\newtheorem{corollary}[theorem]{Corollary}
\newtheorem{lemma}[theorem]{Lemma}
\newtheorem{proposition}[theorem]{Proposition}

\theoremstyle{definition}

\newtheorem{remark}[theorem]{Remark}
\newtheorem{example}[theorem]{Example}

\newcommand{\abs}[1]{\lvert#1\rvert}
\newcommand{\norm}[1]{\lVert#1\rVert}
\newcommand{\bigabs}[1]{\bigl\lvert#1\bigr\rvert}
\newcommand{\bignorm}[1]{\bigl\lVert#1\bigr\rVert}
\newcommand{\Bigabs}[1]{\Bigl\lvert#1\Bigr\rvert}
\newcommand{\Bignorm}[1]{\Bigl\lVert#1\Bigr\rVert}
\renewcommand{\le}{\leqslant}
\renewcommand{\ge}{\geqslant}
\renewcommand{\mid}{\::\:}

\newcommand{\term}[1]{{\textit{\textbf{#1}}}}

\makeatletter
\def\@tvsp{\mathchoice{{}\mkern-4.5mu}{{}\mkern-4.5mu}{{}\mkern-2.5mu}{}}
\def\ltrivert{\left|\@tvsp\left|\@tvsp\left|}
\def\rtrivert{\right|\@tvsp\right|\@tvsp\right|}
\def\bltrivert{\bigl|\@tvsp\bigl|\@tvsp\bigl|}
\def\brtrivert{\bigr|\@tvsp\bigr|\@tvsp\bigr|}
\makeatother
\newcommand{\trinorm}[1]{\ltrivert#1\rtrivert}

\def\api{{\abs{\!\pi\!}}}
\def\ftp{\otimes_{\!\scriptscriptstyle\api}\!}

\def\xyroot{x^{\frac{1}{2}}{y^{\frac{1}{2}}}}
\def\Ioc{I_{\rm oc}}

\DeclareMathOperator{\Span}{span}
\DeclareMathOperator{\cspan}{\overline{span}}
\DeclareMathOperator{\sign}{sign}

\begin{document}
\baselineskip 18pt

\title[The 2-concavification is the diagonal of a tensor square]
      {The 2-concavification of a Banach lattice equals the diagonal 
          of the Fremlin tensor square}

\author[Q.~Bu]{Qingying Bu}
\author[G.~Buskes]{Gerard Buskes}
\author[A.~I.~Popov]{Alexey I. Popov}
\author[A.~Tcaciuc]{Adi Tcaciuc}
\author[V.~G.~Troitsky]{Vladimir G. Troitsky}

\address[Q.~Bu and G.~Buskes]
   {Department of Mathematics, University of Mississippi, University,
     MS 38677-1848. USA}
\email{qbu@olemiss.edu, mmbuskes@olemiss.edu}

\address[A.~I.~Popov]
  {Department of Pure Mathematics, Faculty of Mathematics
   University of Waterloo, Waterloo, Ontario, N2L 3G1. Canada}
\email{a4popov@uwaterloo.ca}

\address[A.~Tcaciuc]{Mathematics and Statistics Department,
   Grant MacEwan University, Edmonton, AB, T5J\,P2P, Canada}
\email{atcaciuc@ualberta.ca}

\address[V.~G.~Troitsky]
  {Department of Mathematical
  and Statistical Sciences, University of Alberta, Edmonton,
  AB, T6G\,2G1. Canada}
\email{troitsky@ualberta.ca}

\thanks{The third, forth, and fifth authors were supported by NSERC}
\keywords{Banach lattice, Fremlin projective tensor product, diagonal
  of tensor square, square of a Banach lattice, concavification}
\subjclass[2010]{Primary: 46B42. Secondary: 46M05, 46B40, 46B45.}

\begin{abstract}
  We investigate the relationship between the diagonal of the Fremlin
  projective tensor product of a Banach lattice $E$ with itself and the
  2-concavification of~$E$.
\end{abstract}

\date{\today}

\maketitle

\section{Introduction and preliminaries}

It is easy to see that the diagonal of the projective tensor product
$\ell_p\otimes_\pi\ell_p$ is isometric to $\ell_{\frac{p}{2}}$ if
$p\ge 2$ and to $\ell_1$ if $1\le p\le 2$. In this paper, we extend
this fact to all Banach lattices. It turns out that the ``right''
tensor product for this problem is the Fremlin projective tensor
product $E\ftp F$ of Banach lattices $E$ and $F$. Given a Banach
lattice~$E$, we define (following \cite{BuskesBu}) the diagonal of
$E\ftp E$ to be the quotient of $E\ftp E$ over the closed order ideal
$\Ioc$ generated by the set $\bigl\{(x\otimes y)\mid x\perp y\bigr\}$.
We study the relationship of this diagonal with the 2-concavification
of~$E$. In the literature, it has been observed (see, e.g.,
~\cite{Lindenstrauss:79}) that the $p$-concavification $E_{(p)}$ of
$E$ is again a Banach lattice when $E$ is $p$-convex. However, without
the $p$-convexity assumption, $E_{(p)}$ is only a semi-normed lattice.
We show that in the case when $E$ is 2-convex, the diagonal of $E\ftp
E$ is lattice isometric to~$E_{(2)}$ and that in general, the diagonal
of $E\ftp E$ is lattice isometric to~$E_{[2]}$, where $E_{[p]}$ is the
completion of $E_{(p)}/\ker\norm{\cdot}_{(p)}$. We also show that if
$E$ satisfies the lower $p$-estimate then $E_{[p]}$ is lattice
isomorphic to an AL-space. In particular, if $E$ satisfies the lower
2-estimate then the diagonal of $E\ftp E$ is lattice isomorphic to an
AL-space.

We consider the special case when $E$ and $F$ are Banach lattices with
(1-unconditional) bases $(e_i)$ and $(f_i)$, respectively. We show
that the double sequence $(e_i\otimes f_j)$ is an unconditional basis
of $E\ftp F$ (while it need not be an unconditional basis for the
Banach space projective tensor product $E\otimes_{\pi} F$,
see~\cite{Kwapien:70}). We also show that in this case $E_{(p)}$ is a
normed lattice and the diagonal of $E\ftp E$ is lattice isometric to
the completion of $E_{(2)}$ via $e_i\otimes e_i\mapsto e_i$. Moreover,
if $(e_i)$ is normalized and $E$ satisfies the lower $p$-estimate then
the completion of $E_{(p)}$ is lattice isomorphic to $\ell_1$. In
particular, if $E$ satisfies the lower 2-estimate then the diagonal of
$E\ftp E$ is lattice isometric to $\ell_1$.

In the rest of this section, we provide some background facts that are
necessary for our exposition.

\subsection{Fremlin tensor product}
We refer the reader to \cite{Fremlin:72,Fremlin:74} for a detailed
original definition of the Fremlin tensor product $E\ftp F$ of two
Banach lattices $E$ and~$F$. However, we will only use a few facts
about $E\ftp F$ that we describe here.

Suppose $E$ and $F$ are two Banach lattices. We write $E\otimes F$ for
their algebraic tensor product; for $x\in E$ and $y\in F$ we write
$x\otimes y$ for the corresponding elementary tensor in $E\otimes F$.
Every element of $E\otimes F$ is a linear combination of elementary
tensors. Let $G$ be another Banach lattice and $\varphi\colon E\times
F\to G$ be a bilinear map. Then $\varphi$ induces a map
$\hat\varphi\colon E\otimes F\to G$ such that $\hat\varphi(x\otimes
y)=\varphi(x,y)$ for all $x\in E$ and $y\in F$. We say that $\varphi$
is continuous if its norm, defined by
\begin{displaymath}
  \norm{\varphi}=\sup\bigl\{\bignorm{\varphi(x,y)}\mid
  \norm{x}\le 1,\ \norm{y}\le 1\bigr\},
\end{displaymath}
is finite. We say that $\varphi$ is \term{positive} if
$\varphi(x,y)\ge 0$ whenever $x,y\ge 0$ and that $\varphi$ is a
\term{lattice bimorphism} if
$\bigabs{\varphi(x,y)}=\varphi\bigl(\abs{x},\abs{y}\bigr)$ for all
$x\in E$ and $y\in F$.  We say that $\varphi$ is \term{orthosymmetric}
if $\varphi(x,y)=0$ whenever $x\perp y$.

For $u\in E\otimes F$, put
\begin{equation}
  \label{norm-phi}
  \norm{u}_\api=
    \sup\norm{\hat\varphi(u)},
\end{equation}
where the supremum is taken over all Banach lattices $G$ and all
positive bilinear maps $\varphi$ from $E\times F$ to $G$ with
$\norm{\varphi}\le 1$. Theorem~1E in \cite[p.~89]{Fremlin:74} proves
that $\norm{\cdot}_\api$ is a norm on $E\otimes F$, and the completion
of $E\otimes F$ with respect to this norm is again a Banach lattice.
We will write $E\ftp F$ for this space and call it the
\term{Fremlin tensor product} of $E$ and~$F$. The Fremlin tensor norm
is a cross norm, i.e., $\norm{x\otimes y}_\api=\norm{x}\cdot\norm{y}$
whenever $x\in E$ and $y\in F$.

\begin{remark}\label{phi-tensor}
  (See 1E(iii) and 1F in \cite[p.~92]{Fremlin:74}.)  Let $E$, $F$, and
  $G$ be Banach lattices. There is a one-to-one norm preserving
  correspondence between continuous positive bilinear maps
  $\varphi\colon E\times F\to G$ and positive operators $T\colon E\ftp
  F\to G$ such that $T(x\otimes y)=\varphi(x,y)$ for all $x\in E$ and
  $y\in F$. We will denote $T=\varphi^\otimes$. Furthermore, $\varphi$
  is a lattice bimorphism if and only if $T$ is a lattice
  homomorphism.
\end{remark}

There is an alternative definition of $E\ftp F$, cf.
\cite[1I]{Fremlin:74} and~\cite[pp.~203-204]{Schaefer:80}. Recall that,
being a dual Banach lattice, $F^*$ is Dedekind complete by
\cite[Theorem~3.49]{Aliprantis:06}, so that the space of regular
operators $L^r(E,F^*)$ is a Banach lattice with respect to the regular
norm $\norm{\cdot}_r$, see \cite[p.~255]{Aliprantis:06}.

\begin{proposition}\label{ftp-repr}
  If $E$ and $F$ are Banach lattices then $E\ftp F$ can be identified
  with a closed sublattice of $L^r(E,F^*)^*$ such that $\langle
  x\otimes y, T\rangle=\langle Tx,y\rangle$ for $x\in E$, $y\in F$,
  and $T\in L^r(E,F^*)$.
\end{proposition}

\begin{proof}
  Consider the map $\alpha\colon h\in(E\ftp F)^*\mapsto T\in L(E,F^*)$
  via $\langle Tx,y\rangle=h(x\otimes y)$. It is easy to see that
  $\alpha$ is one-to one and $T\ge 0$ whenever $h\ge 0$. It follows
  that $\alpha(h)$ is regular for every~$h$.

  Suppose that $0\le T\colon E\to F^*$. The map $\varphi$ defined by
  $\varphi(x,y)=\langle Tx,y\rangle$ is a positive bilinear functional
  on $E\times F$. Also,
  \begin{displaymath}
  \norm{T}
  =\sup\bigl\{\bigabs{\langle Tx,y\rangle}\mid
    \norm{x}\le 1,\,\norm{y}\le 1\bigr\}
   =\sup\bigl\{\bigabs{\varphi(x,y)}\mid
    \norm{x}\le 1,\,\norm{y}\le 1\bigr\}
   =\norm{\varphi}. 
  \end{displaymath}
  By Remark~\ref{phi-tensor}, we can consider $h=\varphi^{\otimes}$,
  then $0\le h\in(E\ftp F)^*$ and $\norm{h}=\norm{\varphi}=\norm{T}$.
  It is easy to see that $T=\alpha(h)$. Hence, the restriction of
  $\alpha$ to the positive cones of $(E\ftp F)^*$ is a bijective
  isometry onto the positive cone of $L(E,F^*)$. It follows by
  \cite[Theorem~2.15]{Aliprantis:06} that $\alpha$ is a latice
  isomorphism between $(E\ftp F)^*$ and $L^r(E,F^*)$. Moreover, if
  $T=\alpha(h)$ for some $h\in(E\ftp F)^*$ then
  $\alpha\bigl(\abs{h}\bigr)=\abs{T}$ yields
  \begin{math}
  \norm{h}=\bignorm{\abs{h}}=\bignorm{\abs{T}}=\norm{T}_r. 
  \end{math}
  It follows that $\alpha$ is a lattice isometry between $(E\ftp F)^*$
  and $L^r(E,F^*)$.  Therefore, $(E\ftp F)^{**}$ is lattice isometric to
  $L^r(E,F^*)^*$.  Since $E\ftp F$ can be viewed as a sublattice of
  $(E\ftp F)^{**}$, it is lattice isometric to a closed sublattice of
  $L^r(E,F^*)^*$.
\end{proof}

\subsection{Functional calculus}
\label{FC}

Given $x$ and $y$ in a Banach lattice~$E$, one would like to
define expressions like $(x^2+y^2)^\frac{1}{2}$ and $\xyroot$ to be
elements of~$E$. This can be done point-wise if $E$ can be represented
as a function space. One could object, however, that the definition
may then depend on the choice of a functional representation.
Theorem~1.d.1 in~\cite{Lindenstrauss:79} (see also \cite{Buskes:91})
proves that there is a unique way to extend all continuous
homogeneous%
\footnote{Recall that a function $f\colon\mathbb R^n\to R$ is called
  \term{homogeneous} if $f(\lambda x_1,\dots,\lambda x_n)=\lambda
  f(x_1,\dots,x_n)$ for any $x_1,\dots,x_n$, and $\lambda\ge 0$.}
functions from $\mathbb R^n$ to $\mathbb R$ to functions from $E^n$ to
$E$ which does not depend on a particular representation of $E$ as a
function space. More precisely, for any $x_1,\dots,x_n\in E$ there
exists a unique lattice homomorphism $\tau$ from the space of all
continuous homogeneous functions on $\mathbb R^n$ to $E$ such that if
$f(t_1,\dots,t_n)=t_i$ then $\tau(f)=x_i$ as $i=1,\dots,n$. We denote
$\tau(f)$ by $f(x_1,\dots,x_n)$. In particular, all identities and
inequalities for homogeneous expressions that are valid in $\mathbb R$
remain valid in~$E$. For example,
\begin{equation}\label{distr}
 (x_1+x_2)^{\frac{1}{2}}y^{\frac{1}{2}}=
 \bigl((x_1^{\frac{1}{2}}y^{\frac{1}{2}})^2+
   (x_2^{\frac{1}{2}}y^{\frac{1}{2}})^2\bigr)^{\frac{1}{2}}
\end{equation}
for every $x_1$, $x_2$, and $y$ in every Banach lattice~$E$.  Note
that, following convention from~\cite[p.~53]{Lindenstrauss:79}, for
$t\in\mathbb R$ and $p>0$, by $t^p$ we mean $\abs{t}^p\sign t$. There
is a certain inconsistency in notation: for example, $t^2$ equals
$t\abs{t}$, not $tt$, so that $(x^2)^{\frac{1}{2}}=x$ while
$(xx)^{\frac{1}{2}}=\abs{x}$. To avoid confusion, we will distinguish
$xx$ from $x^2$ throughout the paper.  Note also that
\begin{equation}
  \label{sq-abs}
  x^{\frac{1}{2}}\abs{x}^{\frac{1}{2}}=x.
\end{equation}

In the following lemma, we collect several standard facts that we will
routinely use. 

\begin{lemma}\label{calc-facts}
  Given any $x,y\in E$ and $p>0$.
  \begin{enumerate}
  \item\label{roots-abs} $\bigabs{\xyroot}
     =\abs{x}^{\frac{1}{2}}\abs{y}^{\frac{1}{2}}$;
  \item\label{roots-prod} $\bignorm{\xyroot}\le
    \norm{x}^{\frac{1}{2}}\norm{y}^{\frac{1}{2}}$;
  \item\label{perp-prod-0}
  If $x\perp y$ then $x^{\xyroot}=0$;
  \item\label{cone-E2} If $x,y\ge 0$  then
    $(x^p+y^p)^{\frac{1}{p}}\ge 0$;
  \item\label{perp-osum}
  If $x\wedge y=0$ then
  \begin{math}
    \bigl(x^p+y^p\bigr)^{\frac{1}{p}}
    =x+y.
  \end{math}
  \end{enumerate}
\end{lemma}

\begin{proof}
  \eqref{roots-abs} follows from the fact that the identity holds for
  real numbers.

  \eqref{roots-prod}
  By Proposition~1.d.2(i) from~\cite{Lindenstrauss:79}, we have
  \begin{math}
    \bignorm{\abs{x}^{\frac{1}{2}}\abs{y}^{\frac{1}{2}}}\le
    \norm{x}^{\frac{1}{2}}\norm{y}^{\frac{1}{2}}.
  \end{math}
  Combining this with~\eqref{roots-abs}, we get the required inequality.

  \eqref{perp-prod-0} follows from the fact that
  \begin{math}
    \bigabs{\xyroot}=
    \bigl(\abs{x}\vee\abs{y}\bigr)^{\frac{1}{2}}
    \bigl(\abs{x}\wedge\abs{y}\bigr)^{\frac{1}{2}}.
  \end{math}

  \eqref{cone-E2} 
  Note that 
  \begin{math}
    \bigl(\abs{x}^p+\abs{y}^p\bigr)^{\frac{1}{p}}\ge 0
  \end{math}
  for every $x,y\in E$ because this inequality is true for real numbers. It
  follows that if $x,y\ge 0$ then
  \begin{math}
    (x^p+y^p)^{\frac{1}{p}}
    =\bigl(\abs{x}^p+\abs{y}^p\bigr)^{\frac{1}{p}}\ge 0.
  \end{math}

  \eqref{perp-osum}
  Again, for every $x,y\in E$ we have
  \begin{math}
    \abs{x}\vee\abs{y}\le\bigl(\abs{x}^p+\abs{y}^p\bigr)^{\frac{1}{p}}
    \le\abs{x}+\abs{y}
  \end{math}
  because this is true for real numbers. But if $x\wedge y=0$ then
  $x,y\ge 0$ and $x\vee y=x+y$.
\end{proof}

A Banach lattice $E$ is said to be $p$-\term{convex} for some $1\le p<\infty$
if there is a constant $M>0$ such that
\begin{math}
 \bignorm{\bigl(\sum_{i=1}^nx_i^p\bigr)^{\frac{1}{p}}}
 \le M\bigl(\sum_{i=1}^n\norm{x_i}^p\bigr)^{\frac{1}{p}}
\end{math}
whenever $x_1,\dots,x_n\in E_+$. Similarly, $E$ is $p$-\term{concave}
if  there is a constant $M>0$ such that
\begin{math}
 \bignorm{\bigl(\sum_{i=1}^nx_i^p\bigr)^{\frac{1}{p}}}
 \ge\frac{1}{M}\bigl(\sum_{i=1}^n\norm{x_i}^p\bigr)^{\frac{1}{p}}
\end{math}
whenever $x_1,\dots,x_n\in E_+$.

A Banach lattice $E$ satisfies the \term{upper}
$p$-\term{estimate} with constant $M$ if
\begin{math}
  \bignorm{\sum_{k=1}^nx_k}
  \le M\bigl(\sum_{k=1}^n\norm{x_k}^p\bigr)^{\frac{1}{p}}
\end{math}
whenever $x_1,\dots,x_n$ are disjoint. Similarly,
$E$ satisfies the \term{lower}
$p$-\term{estimate} with constant $M$ if
\begin{math}
  \bignorm{\sum_{k=1}^nx_k}
  \ge\tfrac{1}{M}\bigl(\sum_{k=1}^n\norm{x_k}^p\bigr)^{\frac{1}{p}}
\end{math}
whenever $x_1,\dots,x_n$ are disjoint.  It follows from
Lemma~\ref{calc-facts}\eqref{perp-osum}
that $p$-convexity implies the
upper $p$-estimate and $p$-concavity implies the lower $p$-estimate.

\section{The concavification of a Banach lattice}

The concavification procedure is motivated by the fact that if
$(x_i)\in\ell_r$ and $1<p<r$, then the sequence $(x_i^p)$ belongs to
$\ell_{\frac{r}{p}}$.

This section is partially based on Section 1.d in~\cite{Lindenstrauss:79}.
Throughout this section, $E$ is a Banach lattice and $p\ge 1$.

We define new vector operations on $E$ via
\begin{math}
  x\oplus y=(x^p+y^p)^{\frac{1}{p}}
\end{math}
and
\begin{math}
  \alpha\odot x=\alpha^{\frac{1}{p}}x  
\end{math}
whenever $x,y\in E$ and $\alpha\in\mathbb R$. (Here again, if $x$,
$y$, or $\alpha$ are not positive then we use the convention described
earlier.) Note that $E$ endowed with these new addition and
multiplication operations and the original order is again a vector lattice by
Lemma~\ref{calc-facts}\eqref{cone-E2}.

Define
\begin{equation}\label{conv-norm}
  \norm{x}_{(p)}=\inf\Bigl\{\sum_{i=1}^n \norm{v_i}^p\mid
        \abs{x}\le v_1\oplus\dots\oplus v_n, v_i\ge 0\Bigr\}.
\end{equation}

\begin{remark}\label{conv-norm-eq}
  Note that being a vector lattice, $(E,\oplus,\odot,\le)$ satisfies
  the Riesz Decomposition Property (see, e.g., Theorem~1.13
  in~\cite{Aliprantis:06}), so that the inequality $\abs{x}\le
  v_1\oplus\dots\oplus v_n$ in~\eqref{conv-norm} can be replaced by
  equality.
\end{remark}

It is easy to see that~\eqref{conv-norm} defines a lattice semi-norm
on $(E,\oplus,\odot,\le)$. This semi-normed vector lattice will be
denoted by~$E_{(p)}$. It is called the \term{$p$-concavification}
of~$E$.  As a partially ordered set, $E_{(p)}$ coincides with~$E$.  We
will see in Examples~\ref{examp:Lp} and~\ref{examp:ellp} that
$\norm{\cdot}_{(p)}$ does not have to be a norm, and when it is a
norm, it need not be complete.

The following fact is standard, we include the proof for completeness.

\begin{proposition}\label{p-conv-BL}
  If $E$ is a $p$-convex Banach lattice then $E_{(p)}$ is a Banach lattice.
\end{proposition}

\begin{proof}
  Suppose that $E$ is $p$-convex with constant~$M$. Given $x\in E$.
  Suppose that 
  \begin{displaymath}
      \abs{x}=v_1\oplus\dots\oplus v_n
    =\bigl(v_1^p+\dots+v_n^p\bigr)^{\frac{1}{p}}
  \end{displaymath}
  for some $v_i\ge 0$. Then
  \begin{math}
    \norm{x}
    \le M\Bigl(\sum_{i=1}^n \norm{v_i}^p\Bigr)^{\frac{1}{p}}.
  \end{math}
  It follows that
  $\frac{1}{M^p}\norm{x}^p\le\norm{x}_{(p)}\le\norm{x}^p$. This yields
  that $\norm{\cdot}_{(p)}$ is a complete norm on~$E_{(p)}$.
\end{proof}

Recall that if $E$ is a Banach lattice and $x>0$, then $x$ is an
\term{atom} in $E$ if $0\le z\le x$ implies that $z$ is a scalar
multiple of~$x$. We say that $E$ is \term{atomic} or \term{discrete}
if for every $z>0$ there exists an atom $x$ such that $0<x\le z$. 

\begin{lemma}\label{atom-norm}
  If $x$ is an atom in a Banach lattice $E$ then $\norm{x}_{(p)}=\norm{x}^p$
\end{lemma}

\begin{proof}
  Take $v_1,\dots,v_n\in E_+$ such that $x=v_1\oplus\dots\oplus v_n$.
  It follows that $0\le v_k\le x$ for each $k=1,\dots,n$, hence
  $v_k=\alpha_k\odot x=\alpha_k^{1/p}x$ for some
  $\alpha_k\in\mathbb R_+$. Also,
  \begin{displaymath}
    x=v_1\oplus\dots\oplus v_n
    =(\alpha_1\odot x)\oplus\dots\oplus(\alpha_n\odot x)
    =(\alpha_1+\dots+\alpha_n)\odot x,
  \end{displaymath}
  so that $\sum_{k=1}^n\alpha_k=1$. It follows that
  \begin{math}
    \sum_{k=1}^n \norm{v_k}^p
    =\sum_{k=1}^n\bignorm{\alpha_k^{\frac{1}{p}}x}^p
    =\norm{x}^p,
  \end{math}
  so that $\norm{x}_{(p)}=\norm{x}^p$.
\end{proof}

\begin{corollary}\label{discr-norm}
  If $E$ is a discrete Banach lattice then $E_{(p)}$ is a normed lattice.
\end{corollary}

\begin{proof}
  Since we know that $\norm{\cdot}_{(p)}$ is a lattice semi-norm
  on~$E_{(p)}$, it suffices to prove that it has trivial kernel.
  Suppose that $y\in E$ with $y\ne 0$. There is an atom $x$ such that
  $0<x\le\abs{y}$. Then
  $\norm{y}_{(p)}\ge\norm{x}_{(p)}=\norm{x}^p>0$.
\end{proof}

\begin{remark}\label{Ep-normed}
  Thus, we know that $E_{(p)}$ is a normed lattice in two important
  special cases: when $E$ is discrete or $p$-convex. It would be
  interesting to find a general characterization of Banach lattices
  $E$ for which $\norm{\cdot}_{(p)}$ is a norm. That is, characterize
  all Banach lattices $E$ such that
  \begin{displaymath}
    \inf\Bigl\{\sum_{i=1}^n \norm{v_i}^p\mid
        x=\bigl(v_1^p+\dots+v_n^p\bigr)^{\frac{1}{p}},
          v_i> 0\Bigr\}>0
  \end{displaymath}
  for every non-zero $x\in E_+$.
\end{remark}

In  general, we can only say that $\norm{\cdot}_{(p)}$ is a
lattice seminorm on~$E_{(p)}$. It follows that its kernel is an ideal,
so that the quotient space $E_{(p)}/\ker\norm{\cdot}_{(p)}$ is a
normed lattice. Denote its completion by~$E_{[p]}$. Clearly, $E_{[p]}$
is a Banach lattice.

\medskip

Let $E$ be a Banach lattice. It is a standard fact (c.f., the proof of
\cite[Lemma~1.b.13]{Lindenstrauss:79}) that if there exists $c>0$ such
that
\begin{math}
    \bignorm{\sum_{k=1}^nx_k}\ge c\sum_{k=1}^n\norm{x_k}
\end{math}
whenever $x_1,\dots,x_n$ are disjoint (that is, if $E$ satisfies the
lower 1-estimate), then $E$ is lattice isomorphic to an $AL$-space.
Indeed, put
\begin{displaymath}
    \trinorm{x}
    =\sup\Bigl\{\sum_{i=1}^n\norm{x_i}\mid
    x_1,\dots,x_n\text{ are positive and disjoint and }\abs{x}=x_1+\dots+x_n\Bigr\}.
\end{displaymath}
It can be easily verified that this is an equivalent norm on $E$ which
makes $E$ into an AL-space (with the same order).

The following lemma establishes that if $E$ satisfies the lower
  $p$-estimate then $E_{(p)}$ satisfies the lower 1-estimate.

\begin{lemma}\label{p-est-norm}
  Suppose that $E$ is a Banach lattice satisfying the lower
  $p$-estimate with constant~$M$. Then
  \begin{math}
    \bignorm{\sum_{k=1}^nx_k}_{(p)}
    \ge\frac{1}{M^p}\sum_{k=1}^n\norm{x_k}_{(p)}
  \end{math}
  whenever $x_1,\dots,x_n$ are disjoint in~$E$.
\end{lemma}

\begin{proof}
  Suppose $x_1,\dots,x_n$ are disjoint in~$E$.
  Since
  \begin{math}
    \bigabs{\sum_{k=1}^nx_k}
    =\sum_{k=1}^n\abs{x_k},
  \end{math}
  we may assume without loss of generality that $x_k\ge 0$ for each~$k$.
  Note that
  \begin{math}
    \sum_{k=1}^nx_k
    =x_1\oplus\dots\oplus x_n
  \end{math}
  by Lemma~\ref{calc-facts}\eqref{perp-osum}.

  We will use~\eqref{conv-norm} and Remark~\ref{conv-norm-eq} to
  estimate $\bignorm{x_1\oplus\dots\oplus x_n}_{(p)}$.  Take
  $u_1,\dots,u_m$ in $E_+$ such that $x_1\oplus\dots\oplus
  x_n=u_1\oplus\dots\oplus u_m$. Since $E_{(p)}$ is a vector lattice,
  by the Riesz Decomposition Property \cite[Theorem~1.20]{Aliprantis:06},
  for each $k=1,\dots,n$ we find $v_{k,1},\dots,v_{k,m}$ in $E_+$ such
  that $x_k=v_{k,1}\oplus\dots\oplus v_{k,m}$ and
  $u_i=v_{1,i}\oplus\dots\oplus v_{n,i}$ for each $i=1,\dots,m$. For
  each $k$ and $i$ we have $0\le v_{k,i}\le x_k$, so that
  $v_{1,i},\dots,v_{n,i}$ are disjoint for every $i$. It follows that
  $u_i=v_{1,i}+\dots+ v_{n,i}$. By the lower $p$-estimate, we get
  \begin{math}
    \norm{u_i}
    \ge\frac{1}{M}\bigl(\sum_{k=1}^n\norm{v_{k,i}}^p)^{\frac{1}{p}},
  \end{math}
  so that
  \begin{math}
    M^p\norm{u_i}^p
    \ge\sum_{k=1}^n\norm{v_{k,i}}^p.
  \end{math}
  For every~$k$, we have
  \begin{math}
    \norm{x_k}_{(p)}\le\sum_{i=1}^m\norm{v_{k,i}}^p,    
  \end{math}
  so that
  \begin{displaymath}
    \sum_{k=1}^n\norm{x_k}_{(p)}
    \le\sum_{k=1}^n\sum_{i=1}^m\norm{v_{k,i}}^p
    \le M^p\sum_{i=1}^m\norm{u_i}^p.
  \end{displaymath}
  Taking the infimum over all $u_1,\dots,u_m$ in $E_+$ such that $x_1\oplus\dots\oplus
  x_n=u_1\oplus\dots\oplus u_m$, we get the required inequality.
\end{proof}

\begin{theorem}\label{p-est-AL}
  If a Banach lattice $E$ satisfies the lower $p$-estimate with constant
  $M$ then $E_{[p]}$ is lattice isomorphic to an
  AL-space. Furthermore, if $M=1$ then $E_{[p]}$ is an AL-space.
\end{theorem}

\begin{proof}
  Suppose that $E$ satisfies a lower $p$-estimate with constant~$M$.
  Applying Lemma~\ref{p-est-norm}, we have
  \begin{math}
    M^p\bignorm{\sum_{k=1}^nx_k}_{(p)}
    \ge\sum_{k=1}^n\norm{x_k}_{(p)}
  \end{math}
  whenever $x_1,\dots,x_n$ are disjoint in~$E$. It is easy to see that
  this inequality remains valid in $E_{(p)}/\ker\norm{\cdot}_{(p)}$
  and,
  furthermore, in~$E_{[p]}$.
\end{proof}

\section{Main results}

Let $E$ be a Banach lattice. Let $\Ioc$ be the norm closed ideal
generated in $E\ftp E$ by the elements of the form $x\otimes y$ where
$x\perp y$ (without loss of generality, we may also assume that $x$
and $y$ are positive). We can view $\Ioc$ as the set of all
``off-diagonal'' elements of $E\ftp E$. Therefore,
following~\cite{BuskesBu}, we think of $(E\ftp E)/\Ioc$ as the
diagonal of $E\ftp E$. We claim that this space is lattice isometric
to~$E_{[2]}$.

\begin{theorem}\label{general-gen}
  Suppose that $E$ is a Banach lattice. Then there exists a surjective
  lattice isometry $T\colon E_{[2]}\to(E\ftp E)/\Ioc$ such that
  $T\bigl(x+\ker\norm{\cdot}_{(2)}\bigr)=x\otimes\abs{x}+\Ioc$ for
  each $x\in E$.
\end{theorem}

\begin{proof}
  Define a map $\varphi\colon E\times E\to E_{(2)}$ by
  $\varphi(x,y)=\xyroot$. By the nature of the vector
  operations in~$E_{(2)}$, this map is bilinear. Indeed,
  \[
    \varphi(\lambda x,y)=(\lambda x)^{\frac{1}{2}}y^{\frac{1}{2}}
      =\lambda^{\frac{1}{2}}\xyroot
      =\lambda\odot(\xyroot)
      =\lambda\odot \varphi(x,y).
  \]  
  Similarly, $\varphi(x,\lambda y)=\lambda\odot \varphi(x,y)$. Also,
  $\varphi(x_1+x_2,y)=\varphi(x_1,y)\oplus \varphi(x_2,y)$ by~\eqref{distr}; we obtain
  $\varphi(x,y_1+y_2)=\varphi(x,y_1)\oplus \varphi(x,y_2)$ in a similar fashion.  For
  any $x,y\in E$ we have by Lemma~\ref{calc-facts}\eqref{roots-prod}
  \[
    \bignorm{\varphi(x,y)}_{(2)}=\bignorm{\xyroot}_{(2)}
     \le \bignorm{\xyroot}^2
     \le \Bigl(\norm{x}^{\frac{1}{2}}\norm{y}^{\frac{1}{2}}\Bigr)^2
     =\norm{x}\norm{y},
  \]
  so that $\norm{\varphi}\le 1$. Clearly, $\varphi$ is a continuous
  lattice bimorphism; it is orthosymmetric by
  Lemma~\ref{calc-facts}\eqref{perp-prod-0}.

  Put $N=\ker\norm{\cdot}_{(2)}$ and let $r\colon E_{(2)}\to
  E_{(2)}/N$ be the canonical quotient map. Also, let $i\colon E_{(2)}/N\to
  E_{[2]}$ be the natural inclusion map. Consider the map
  $(ir\varphi)^\otimes\colon E\ftp E\to E_{[2]}$ as in
  Remark~\ref{phi-tensor} (see Fugure~\ref{diagram-gen}); then
  $(ir\varphi)^\otimes$ is a lattice homomorphism and
  $\norm{(ir\varphi)^\otimes}\le 1$.  Note that if $x\perp y$ then
  $(ir\varphi)^\otimes(x\otimes y)=ir\varphi(x,y)=0$.  Since
  $(ir\varphi)^\otimes$ is positive, it vanishes on~$\Ioc$. Consider
  the quotient space $(E\ftp E)/\Ioc$; let
  \begin{math}
  q\colon E\ftp E\to(E\ftp E)/\Ioc  
  \end{math}
  be the canonical
  quotient map.  Since $\Ioc\subseteq\ker(ir\varphi)^\otimes$, we can
  consider the induced map $\widetilde{(ir\varphi)^\otimes}\colon
  (E\ftp E)/\Ioc\to E_{[2]}$ such that
  \begin{math}
    \widetilde{(ir\varphi)^\otimes}q=(ir\varphi)^\otimes.
  \end{math}

  \begin{figure}[!htb]\caption{}\label{diagram-gen}
    \[\xymatrix@=80pt{
    E\times E \ar[r]^\varphi \ar[d]_{\otimes}  & E_{(2)} 
         \ar[ddl]^{S} \ar[r]^r  & 
         E_{(2)}/N \ar[ddll]^{\widetilde{S}} \ar[r]^i &
         E_{[2]} \ar@{.>}@/^2pc/[ddlll]^T  \\
    E\ftp E \ar[urrr]^(0.4){(ir\varphi)^{\otimes}} \ar[d]_q &&&\\
    (E\ftp E)/\Ioc \ar[uurrr]|*+<4pt>{\scriptstyle{\widetilde{(ir\varphi)^\otimes}}} &&
    }
    \]
  \end{figure}

  Consider the map $q\otimes$ from $E\times E$ to
  $(E\ftp E)/\Ioc$. This map is clearly bilinear and
  orthosymmetric. Therefore, by Theorem~9(ii)
  of~\cite{Buskes:01}, there exists a lattice homomorphism $S\colon
  E_{(2)}\to (E\ftp E)/\Ioc$ such that $q\otimes=S\varphi$.
  Note that for each $x,y\in E$ we have
  \begin{equation}
     \label{Sxy-gen}
     S(\xyroot)=S\varphi(x,y)=q\otimes(x,y)=
     x\otimes y+\Ioc.
  \end{equation}
  In particular, taking $y=\abs{x}$, we get
  $Sx=x\otimes\abs{x}+\Ioc$.

  We claim that $\norm{Sx}\le\norm{x}_{(2)}$ for each $x\in E$. Indeed,
  take $v_1,\dots,v_n\in E_+$ such that
  $\abs{x}=v_1\oplus\dots\oplus v_n$. Since $S$ is a lattice homomorphism,
  we have
  \[
    \abs{Sx}=S\abs{x}=Sv_1+\dots+Sv_n
    =v_1\otimes\abs{v_1}+\dots+v_n\otimes\abs{v_n}+\Ioc.
  \]
  By the definition of a quotient norm,
  \[
    \norm{Sx}\le\Bignorm{\sum_{i=1}^nv_i\otimes\abs{v_i}}_\api\le
    \sum_{i=1}^n\bignorm{v_i\otimes\abs{v_i}}_\api=
    \sum_{i=1}^n\norm{v_i}^2
  \]
  because $\norm{\cdot}_\api$ is a cross-norm.  It follows now
  from~\eqref{conv-norm} that $\norm{Sx}\le\norm{x}_{(2)}$.

  In particular, $N\subseteq\ker S$. It follows that $S$ induces a
  lattice homomorphism $\widetilde{S}\colon E_{(2)}/N\to(E\ftp
  E)/\Ioc$ such that $S=\widetilde{S}r$. We now show that $\widetilde
  S$ is an isometry.  For any $x\in E$ we have
  \begin{math}
    \norm{\widetilde{S}(x+N)}
    =\norm{Sx}\le\norm{x}_{(2)}
    =\norm{x+N},
  \end{math}
  so that $\norm{\widetilde S}\le 1$. On the other hand,
  for every $v\in \Ioc$ we have
  $(ir\varphi)^\otimes(v)=0$, so that $(ir\varphi)^\otimes\bigl(x\otimes
  \abs{x}+v\bigr)=r\varphi\bigl(x,\abs{x}\bigr)=rx=x+N$ by~\eqref{sq-abs}. Since
  $\norm{(ir\varphi)^\otimes}\le 1$, we get
  $\norm{x+N}\le\bignorm{x\otimes\abs{x}+v}_\api$. Taking infimum over
  all $v\in \Ioc$, we get $\norm{x+N}\le\norm{Sx}=\norm{\widetilde
    S(x+N)}$. Therefore, $\widetilde S$ is an isometry. It follows that
  $\widetilde S$ extends to a lattice isometry
  $T\colon E_{[2]}\to(E\ftp E)/\Ioc$. Note that
  $T(x+N)=Sx=x\otimes\abs{x}+\Ioc$ for each $x\in E$.

  We claim that $T$ is the inverse of
  $\widetilde{(ir\varphi)^\otimes}$. Indeed, for every $x\in E$
  we have
  \begin{displaymath}
     \widetilde{(ir\varphi)^\otimes}T(x+N)=
     \widetilde{(ir\varphi)^\otimes}\bigl(x\otimes\abs{x}+\Ioc\bigr)
     =(ir\varphi)^\otimes\bigl(x\otimes\abs{x}\bigr)
     =ir\varphi\bigl(x,\abs{x}\bigr)=irx=x+N
  \end{displaymath}
  by~\eqref{sq-abs}. This means that $\widetilde{(ir\varphi)^\otimes}T$
  is the identity on $E_{(2)}/N$ and, therefore, on~$E_{[2]}$.
  On the other hand, for each $x,y\in E$ it follows
  from~\eqref{Sxy-gen} that
  \begin{multline*}
    T\widetilde{(ir\varphi)^\otimes}(x\otimes y+\Ioc)
    =T(ir\varphi)^\otimes(x\otimes y)
    =Tir\varphi(x,y)\\
    =Tr(x^{\frac{1}{2}}y^{\frac{1}{2}})
    =\widetilde Sr(x^{\frac{1}{2}}y^{\frac{1}{2}})
    =S(x^{\frac{1}{2}}y^{\frac{1}{2}})
    =x\otimes y+\Ioc.
  \end{multline*}
  Hence, $T\widetilde{(ir\varphi)^\otimes}$ is the identity on $q(E\otimes
  E)$.  Since $E\otimes E$ is dense in $E\ftp E$ then $q(E\otimes E)$
  is dense in $(E\ftp E)/\Ioc$, so that
  $T\widetilde{(ir\varphi)^\otimes}$ is the identity on $(E\ftp
  E)/\Ioc$.  Therefore, $T$ is the inverse of
  $\widetilde{(ir\varphi)^\otimes}$. It follows that $T$ is onto.
\end{proof}

Recall that if $E$ is discrete then $\norm{\cdot}_{(2)}$ is a norm by
Corollary~\ref{discr-norm}, so that $E_{(2)}$ is a normed lattice and
$E_{[2]}$ equals~$\overline{E_{(2)}}$, the completion of
$E_{(2)}$; if $E$ is 2-convex then $E_{(2)}$ is a Banach lattice by
Proposition~\ref{p-conv-BL}; in this case $E_{[2]}=E_{(2)}$.

\begin{corollary}\label{general-normed}
  Suppose that $E$ is Banach lattice. If $E_{(2)}$ is a normed
  lattice then it is lattice isometric to a dense sublattice of
  $(E\ftp E)/\Ioc$ via $x\in E_{(2)}\mapsto x\otimes\abs{x}+\Ioc$.
\end{corollary}

\begin{corollary}\label{general-BL}
  Suppose that $E$ is a Banach lattice such that $E_{(2)}$ is also a
  Banach lattice. Then the map $T\colon E_{(2)}\to(E\ftp E)/\Ioc$
  given by $Tx=x\otimes\abs{x}+\Ioc$ is a surjective linear lattice
  isometry.
\end{corollary}

\begin{remark}
  Theorem~\ref{general-gen} provides a new characterization of the
  ideal~$\Ioc$. It was observed in the proof of
  Theorem~\ref{general-gen} that $(ir\varphi)^\otimes$ vanishes on
  $\Ioc$ and the induced map $\widetilde{(ir\varphi)^\otimes}$ on
  $(E\ftp E)/\Ioc$ is a bijection; hence
  $\Ioc=\ker(ir\varphi)^\otimes$. This can be used to easily verify
  whether certain elements belong to~$\Ioc$. For example, it
  follows from $\varphi\bigl(\xyroot,\xyroot\bigr)=\varphi(x,y)$ that
   \begin{math}
     x\otimes y-\bigl(\xyroot\bigr)\otimes\bigl(\xyroot\bigr)
   \end{math}
  is in $\ker(ir\varphi)^\otimes$ and, hence, in $\Ioc$ for every
  $x,y\in E$.

  Similarly, one can check that $x\otimes y-y\otimes x\in\Ioc$ for all
  $x,y\in E$. Let $Z$ be the closed sublattice of $E\ftp E$ generated
  by vectors of the form $x\otimes y-y\otimes x$. We refer to $Z$ as
  the \term{antisymmetric part} of $E\ftp E$. This yields
  $Z\subseteq\Ioc$ (this inclusion also follows from Proposition~4.33
  of~\cite{Loane:thesis}, obtained there by very different means).
\end{remark}

\begin{remark}\label{Ioc-m}
  Suppose that $E$ is such that $E_{(2)}$ is a Banach lattice. Then we
  can identify $T^{-1}$ in Corollary~\ref{general-BL}.  Indeed, in
  this case, the maps $i$ and $r$ in the proof of
  Theorem~\ref{general-gen} are just the identity maps, so that
  $T^{-1}=\widetilde{\varphi^\otimes}$ where $\varphi(x,y)=\xyroot$
  (see Figure~\ref{diagram}).
  \begin{figure}[!htb]\caption{}\label{diagram}
     \[\xymatrix{
     E\times E \ar[r]^\varphi \ar[d]_{\otimes}  & E_{(2)} \ar@{.>}@/^1pc/[ddl]^{S=T}\\
     E\ftp E \ar[ur]^{\varphi^{\otimes}} \ar[d]_q &\\
     (E\ftp E)/\Ioc \ar[uur]|*+<4pt>{\scriptstyle{\widetilde{\varphi^\otimes}}} &
     }
     \]
  \end{figure}
   Furthermore, in this case, we have
  $\Ioc=\ker \varphi^\otimes$.
\end{remark}

\begin{remark}\label{diag-ExE}
  Again, suppose that $E$ is such that $E_{(2)}$ is a Banach lattice.
  It follows from Corollary~\ref{general-BL} that every equivalence
  class in $(E\ftp E)/\Ioc$ contains a representative of the form
  $x\otimes\abs{x}$ for some $x\in E$.  Therefore, $q(E\otimes
  E)=(E\ftp E)/\Ioc$, where $q\colon E\ftp E\to(E\ftp E)/\Ioc$ is the
  canonical quotient map.  In other words, the elements of $E\otimes
  E$ (and even elementary tensor products) are sufficient to ``capture
  all of the diagonal'' in $E\ftp E$.

  As usual, one can identify $q(E\otimes E)$ with the quotient of
  $E\otimes E$ over $\Ioc$ or, more precisely, with $(E\otimes
  E)/\bigl((E\otimes E)\cap \Ioc\bigr)$, where $E\otimes E$ is viewed
  as a (non-closed) subspace of $E\ftp E$. Therefore,
  \begin{equation}\label{di-ExE}
    (E\ftp E)/\Ioc=
     (E\otimes E)/\bigl((E\otimes E)\cap \Ioc\bigr)
  \end{equation}
\end{remark}

Combining Theorems~\ref{p-est-AL} and~\ref{general-gen}, we
immediately get the following.

\begin{corollary}
  Suppose that $E$ is a Banach lattice satisfying the lower 2-estimate
  with constant~$M$. Then $(E\ftp E)/\Ioc$ is lattice isomorphic to an
  AL-space. If $M=1$ then $(E\ftp E)/\Ioc$ is an AL-space.
\end{corollary}




\section{Function spaces}\label{functions}

In this section, we consider the case when $E$ is a K\"othe space on a
$\sigma$-finite measure space $(\Omega,\Sigma,\mu)$ as in
\cite[Definition~1.b.17]{Lindenstrauss:79}. That is, $E$ is contained
in the space $L_0(\Omega)$ of all measurable functions on $\Omega$
such that $E$ contains the characteristic functions of all sets of
finite measure and if $f\in E$, $g\in L_0(\Omega)$ and
$\abs{g}\le\abs{f}$ then $g\in E$ and $\norm{g}\le\norm{f}$.

It is easy to see that in a K\"othe space, the functional calculus
map~$\tau$, described in Subsection~\ref{FC}, agrees with almost
everywhere pointwise operations. Indeed, fix $x_1,\dots,x_n$ in $E$
and let $h\colon\mathbb R^n\to\mathbb R$ be a homogeneous continuous
function. It is easy to see that
\begin{displaymath}
  \bigabs{h(t_1,\dots,t_n)}\le M\max\limits_{1\le i\le n}\abs{t_i}
\end{displaymath}
for all $t_1,\dots,t_n\in\mathbb R$, where
\begin{displaymath}
  M=\max\bigl\{\abs{h(t_1,\dots,t_n)}\mid
  \max\limits_{1\le i\le n}\abs{t_i}=1\bigr\}.
\end{displaymath}
It follows that
\begin{displaymath}
  \Bigabs{h\bigl(x_1(\omega),\dots,x_n(\omega)\bigr)}\le
  M\max\limits_{1\le i\le n}\bigabs{x_i(\omega)}
\end{displaymath}
for all $\omega\in\Omega$, so that the usual composition function
$h(x_1,\dots,x_n)$ defined a.e.\ by 
\begin{displaymath}
  h(x_1,\dots,x_n)(\omega)=h\bigl(x_1(\omega),\dots,x_n(\omega)\bigr)
\end{displaymath}
satisfies
\begin{displaymath}
  \bigabs{h(x_1,\dots,x_n)}\le M\bigvee\limits_{1\le i\le n}\abs{x_i}
  \mbox{ a.e.;}
\end{displaymath}
it follows that $h(x_1,\dots,x_n)\in E$. Thus, almost everywhere
pointwise operations define a functional calculus on~$E$.  It follows
from the uniqueness of functional calculus that this functional
calculus agrees with $\tau$
\footnote{The same argument shows that on $C(K)$-spaces, $\tau$ agrees
  with the pointwise operations.}.

We proceed with a functional representation of $E_{(2)}$ (see,
e.g.,~\cite{Buskes:01} or \cite[p.~30]{Johnson:01}). The square of $E$
is defined via $E^2=\{x^2\mid x\in E\}$, where, again, by $x^2$ we
really mean $x\abs{x}$ and the product is defined a.e..  Note that the
map $S\colon x\in E_{(2)}\mapsto x^2\in E^2$ is a bijection.  In view
of this, we may transfer the Banach lattice structure from $E_{(2)}$
to~$E^2$. In particular, with this identification, $E^2$ is a vector
space. The main advantage of this approach is that addition and scalar
multiplication in $E^2$ are defined a.e.\ pointwise (the vector
operations on $E_{(2)}$ were defined exactly this way):
\[
  S(x\oplus y)=(x\oplus y)^2=\bigl((x^2+y^2)^{\frac{1}{2}}\bigr)^2=
    x^2+y^2=S(x)+S(y)
\]
and
\[
  S(\lambda\odot x)=(\lambda\odot x)^2=
   \bigl(\lambda^{\frac{1}{2}}x\bigr)^2=\lambda x^2=\lambda S(x).
\]
Observe, also, that if $x,y\in E$ then the function $xy$ is in~$E^2$.
Indeed, $\xyroot\in E$, so that
\begin{math}
  E^2\ni S\bigl(\xyroot\bigr)
  =\bigl(\xyroot\bigr)^2=xy.
\end{math}

In view of this construction, we can replace $E_{(2)}$ with $E^2$ in
the preceding section. In particular, instead of the map
$\varphi\colon E\times E\to E_{(2)}$ defined by $\varphi(x,y)=\xyroot$
in Remark~\ref{Ioc-m}, we can consider the corresponding map $m\colon
E\times E\to E^2$ defined by $m(x,y)=xy$. This map is obviously a
continuous orthosymmetric lattice bimorphism.

Suppose now that $E^2$ is a Banach lattice (for example, $E$ is
2-convex). Then the diagram in Figure~\ref{diagram} in
Corollary~\ref{general-BL} and Remark~\ref{Ioc-m}
becomes the diagram in Figure~\ref{m-diagram}. 
\begin{figure}[!htb]\caption{}\label{m-diagram}
  \[\xymatrix{
  E\times E \ar[r]^m \ar[d]_{\otimes}  & E^2 \ar@{.>}@/^1pc/[ddl]^T\\
  E\ftp E \ar[ur]^{m^{\otimes}} \ar[d]_q &\\
  (E\ftp E)/\Ioc \ar[uur]|*+<4pt>{\scriptstyle{\widetilde{m^\otimes}}} &
  }
  \]
\end{figure}

For $x,y\in E$, their elementary tensor product $x\otimes y$ can be
viewed as a function on $\Omega^2$ via $(x\otimes y)(s,t)=x(s)y(t)$
for $s,t\in\Omega$. This way, $E\otimes E$ is a subset of
$L_0(\Omega^2)$. We do not know whether $E\ftp E$ can
still be viewed as a sublattice of $L_0(\Omega^2)$, but this is
definitely the case in many important special cases.

Let $D$ be the diagonal of~$\Omega^2$, that is, $D=\bigl\{(s,s)\mid
s\in\Omega\bigr\}$. Of course, the map $s\to (s,s)$ is a bijection
between $\Omega$ and~$D$, so that we can view $D$ as a copy
of~$\Omega$. For an arbitrary function $u$ in $L_0(\Omega^2)$, one
cannot really consider the restriction of $u$ to $D$ because $D$ may
have measure zero in~$\Omega^2$. However, such a restriction may be
defined for elementary tensors via $(x\otimes y)(s,s)=x(s)y(s)$, which
is defined a.e. on~$\Omega$. That is, the restriction of $x\otimes y$
to $D$ is exactly $xy=m(x,y)=m^\otimes(x\otimes y)$ (as we identify
$D$ with $\Omega$). Extending this by linearity to $E\otimes E$, we
can view $m^\otimes$ on $E\otimes E$ (or even on $E\ftp E$) as the
\textit{restriction to the diagonal} map. Note that, in view of
Remark~\ref{diag-ExE} and, in particular, \eqref{di-ExE}, the space
$E\otimes E$ is sufficient to capture the diagonal part of $E\ftp E$.
Furthermore, for $u\in E\otimes E$ we have $u\in \Ioc$ iff
$m^\otimes(u)=0$ iff $u$ vanishes a.e.\ on the diagonal. It follows
that the both quotient spaces in~\eqref{di-ExE} can be viewed as the
space of the restrictions of the functions in $E\otimes E$ to~$D$.
Therefore, in the case of K\"othe spaces, Corollary~\ref{general-BL}
says that the restrictions of the elements of $E\otimes E$ (or $E\ftp
E$) to the diagonal are exactly the functions in $E^2$ (again, we
identify the diagonal with $\Omega$). Moreover, the norm of the
restriction (that is, the quotient norm from \eqref{di-ExE}) is the
same as its $E^2$ norm.

\begin{example}
\label{examp:Lp}
If $E=L_p$ for $1\le p<\infty$ then $E^2$ as a vector lattice
coincides with $L_{\frac{p}{2}}$. In the case $p\ge 2$, $E$ is
2-convex and hence $(E\ftp
E)/\Ioc=E_{[2]}=E_{(2)}=L_{\frac{p}{2}}$.  In the case $1\le p<2$, the
vector lattice $L_{\frac{p}{2}}$ (and, therefore, $E_{(2)}$) admits no
non-trivial positive functionals by, e.g.,
\cite[Theorem~5.24]{Aliprantis:03}. Note that every positive
functional $f$ on $E_{[2]}$ gives rise to a positive functional
$f\circ q$ on~$E_{(2)}$, where
\begin{math}
  q\colon E_{(2)}\to E_{(2)}/\ker\norm{\cdot}_{(2)}  
\end{math}
is the canonical quotient map. It follows that $E_{[2]}^*$ is trivial,
and so is~$E_{[2]}$. Hence $(E\ftp E)/\Ioc=E_{[2]}=\{0\}$, which is a
trivial AL-space.
\end{example}

\begin{example}
  Let $E=C[0,1]$. In this case, $E^2=E$. Also, $E\ftp E=C[0,1]^2$ by
  Corollary~3F of~\cite{Fremlin:74}. As before, we put $m(x,y)=xy$ for
  $x,y\in E$. In this case, the map $m^\otimes$ on $E\otimes E$ and,
  therefore, on $E\ftp E$, is the restriction to the diagonal, so that
  $\Ioc$ consists of those functions that vanish on the diagonal,
  while $(E\ftp E)/\Ioc$ is the space of the restrictions of the
  functions in $C[0,1]^2$ to the diagonal, which, naturally, can again
  be identified with $C[0,1]$.
\end{example}

\section{Banach lattices with a basis}

By a \term{Banach lattice with a basis} we mean a Banach lattice where
the order is defined by a basis. That is, $E$ has a (Schauder) basis
$(e_i)$ such that a vector $x=\sum_{i=1}^\infty x_ie_i$ is positive iff
$x_i\ge 0$ for all~$i$.  It follows that the basis $(e_i)$ is
1-unconditional. The converse is also true: every Banach space with a
1-unconditional basis is a Banach lattice in the induced order. It is
clear that every Banach lattice with a basis is discrete.

\subsection{Concavification of a Banach lattice with a basis}

Since $E$ is a K\"othe space, its continuous homogeneous functional
calculus in $E$ is coordinate-wise. For example, if
$x=\sum_{i=1}^{\infty}x_ie_i$ and $y=\sum_{i=1}^{\infty}y_ie_i$ then
\begin{displaymath}
 \xyroot=
  \sum_{i=1}^{\infty}x_i^{\frac{1}{2}}y_i^{\frac{1}{2}}e_i
 \qquad\mbox{and}\qquad
 (x^p+y^p)^{\frac{1}{p}}=
  \sum_{i=1}^{\infty}(x_i^p+y_i^p)^{\frac{1}{p}}e_i. 
\end{displaymath}
As before, we use the conventions $t^p=\abs{t}^p\sign t$ here for
$t,p\in\mathbb R$. 

Next, we fix $p\ge 1$ and consider~$E_{(p)}$. Since $E$ is discrete,
$E_{(p)}$ is a normed lattice by Corollary~\ref{discr-norm}. Hence, in
this case, $E_{[p]}$ equals~$\overline{E_{(p)}}$, the completion
of~$E_{(p)}$. Since $(e_i)$ is disjoint in~$E$, it follows from
Lemma~\ref{calc-facts}\eqref{perp-osum} that
$x_1e_1+\dots+x_ne_n=x_1^p\odot e_1\oplus\dots\oplus x_n^p\odot e_n$.

\begin{lemma}\label{Ep-basis}
  Suppose that $E$ is a Banach lattice with a basis $(e_i)$. Then
  \begin{enumerate}
  \item\label{same-norm} $\norm{e_i}_{(p)}=\norm{e_i}^p$ for each $i$;
  \item\label{espans} if $x=\sum_{i=1}^\infty x_ie_i$ in $E$ then
    $x=\oplus{\rm -}\!\sum_{i=1}^\infty x_i^p\odot e_i$ in $E_{(p)}$; in
    particular, the series converges in $E_{(p)}$;
  \item\label{ei-basis-Ep} $(e_i)$ is a 1-unconditional basis of~$\overline{E_{(p)}}$.
  \end{enumerate}
\end{lemma}

\begin{proof}
  \eqref{same-norm} follows immediately from Lemma~\ref{atom-norm}. To
  prove~\eqref{espans}, suppose that
  $x=\sum_{i=1}^\infty x_ie_i$ in~$E$. For each~$n$, we can write
  $x=u_n+v_n=u_n\oplus v_n$ where $u_n=\sum_{i=1}^n x_ie_i$ and
  $v_n=\sum_{i=n+1}^\infty x_ie_i$. Note that $\norm{v_n}\to 0$ and
   $u_n=\oplus{\rm
    -}\!\sum_{i=1}^n x_i^p\odot e_i$. Therefore,
  \begin{displaymath}
    \Bignorm{x\ominus\bigl(\oplus{\rm -}\!\sum_{i=1}^n
       x_i^p\odot e_i\bigr)}_{(p)}
    =\norm{x\ominus u_n}_{(p)}
    =\norm{v_n}_{(p)}
    \le\norm{v_n}^p\to 0.
  \end{displaymath}
  This proves~\eqref{espans}. It follows from~\eqref{espans} that the
  closed linear span of $(e_i)$ is dense in $E_{(p)}$ and, therefore,
  in~$\overline{E_{(p)}}$. Since the sequence $(e_i)$ remains disjoint
  in~$\overline{E_{(p)}}$, this yields~\eqref{ei-basis-Ep}.
\end{proof}

\begin{proposition}\label{p-est-basis}
  Suppose that $E$ is a Banach lattice with a normalized basis. If $E$
  satisfies the lower $p$-estimate with constant $M$ then
  $\overline{E_{(p)}}$ is lattice isomorphic (isometric if $M=1$) to
  $\ell_1$ via $(x_i)\in\ell_1\mapsto\sum_{i=1}^\infty x_i\odot e_i\in E_{(p)}$.
\end{proposition}

\begin{proof}
  Let $x\in E$ such that $x=\sum_{i=1}^n x_i\odot e_i=\sum_{i=1}^n
  x_i^{1/p}e_i$. It follows from Lemma~\ref{p-est-norm} that
  \begin{displaymath}
    \norm{x}_{(p)}
    \ge\tfrac{1}{M^p}\sum_{i=1}^n\norm{x_i\odot e_i}_{(p)}
    =\tfrac{1}{M^p}\sum_{i=1}^n\abs{x_i}.
  \end{displaymath}
  On the other hand, by the triangle inequality, we have
  \begin{math}
    \norm{x}_{(p)}
    \le\sum_{i=1}^n\norm{x_i\odot e_i}_{(p)}
    =\sum_{i=1}^n\abs{x_i}.
  \end{math}
\end{proof}

\subsection{Fremlin tensor product of Banach lattices with bases}

Given Banach spaces $E$ and $F$ with bases $(e_i)$ and $(f_i)$,
respectively, then the double sequence $(e_i\otimes f_j)$ is a basis
for the Banach space projective tensor product $E\otimes_{\pi}F$,
see~\cite{Gelbaum:61}. However, even if these respective bases are
unconditional then $(e_i\otimes f_j)$ is not necessarily an
unconditional basis for $E\otimes_{\pi}F$. Indeed, it was shown
in~\cite{Kwapien:70} that the Banach space projective tensor product
$\ell_p\otimes_{\pi}\ell_q$ with $1/p + 1/q \le 1$ does not have an
unconditional basis.

Recall that if $E$ is a Banach lattice with a basis then the basis is
automatically 1-unconditional.

\begin{lemma}
  Suppose that $E$ and $F$ are Banach lattices with bases, $(e_i)$ and
  $(f_j)$, respectively. Then the double sequence $(e_i\otimes
  f_j)_{i,j}$ is disjoint in $E\ftp F$. Moreover, this sequence is a
  1-unconditional basis of $E\ftp F$ (under any enumeration).
\end{lemma}

\begin{proof}
  First, we will show that $(e_i\otimes f_j)\perp(e_k\otimes f_l)$
  provided $(i,j)\ne(k,l)$. Using Proposition~\ref{ftp-repr}, we
  consider $E\ftp F$ as a sublattice of $L^r(E,F^*)^*$. It suffices to
  show that
  \begin{displaymath}
    \bigl\langle(e_i\otimes f_j)\wedge(e_k\otimes f_l),T\bigr\rangle=0
  \end{displaymath}
  for every positive $T\colon E\to F^*$. By
  \cite[Theorem~3.49]{Aliprantis:06},
  \begin{equation}
    \label{wedge}
    \bigl\langle(e_i\otimes f_j)\wedge(e_k\otimes f_l),T\bigr\rangle=
    \inf\limits_{0\le S\le T}\bigl\{
       (e_i\otimes f_j)(S)+(e_k\otimes f_l)(T-S)\bigr\}.
  \end{equation}
  Put $c=\langle Te_k,f_l\rangle$ and define $S\colon E\to F^*$ via
  $S=ce^*_k\otimes f^*_l$, where $e^*_k$ and $f^*_l$ are the
  appropriate bi-orthogonal functionals. That is, for $x\in E$ we have
  $Sx=ce^*_k(x)f^*_l$.  Clearly, $S\ge 0$. We will show that $S\le T$.
  It suffices to show that $Se_m\le Te_m$ for every~$m$. But if $m\ne
  k$ then $Se_m=0\le Te_m$. It is left to prove that $Se_k\le Te_k$.
  Note that $Se_k=cf^*_l$. It suffices to show that $\langle
  Se_k,f_n\rangle\le\langle Te_k,f_n\rangle$ for all~$n$. But this is
  true because
  \begin{math}
    \langle Se_k,f_n\rangle=cf^*_l(f_n)=0,
  \end{math}
  when $n\ne l$, and
  \begin{math}
    \langle Se_k,f_l\rangle=cf^*_l(f_l)=c=\langle Te_k,f_l\rangle
  \end{math}
  Now substituting this $S$ into~\eqref{wedge}, we get
  \begin{multline*}
    (e_i\otimes f_j)(S)+(e_k\otimes f_l)(T-S)=
    ce_k^*(e_i)f^*_l(f_j)+
      \langle Te_k,f_l\rangle-\langle Se_k,f_l\rangle
      =0+c-c=0
  \end{multline*}
  because $(i,j)\ne(k,l)$.

  Being a disjoint sequence in a Banach lattice, $(e_i\otimes
  f_j)_{i,j}$ is a 1-unconditional basic sequence. It is left to show
  that its closed span is all of $E\ftp F$.  Take $x\in E$ and $y\in
  F$ with $\norm{x},\norm{y}\le 1$. Given any $\varepsilon\in(0,1)$, we can
  find basis projections $P$ and $Q$ on $E$ and~$F$, respectively,
  such that $x_0=Px$ and $y_0=Qy$ satisfy $\norm{x-x_0}<\varepsilon$
  and $\norm{y-y_0}<\varepsilon$. It follows that
  \begin{multline*}
   \norm{x\otimes y-x_0\otimes y_0}_\api=
   \norm{x_0\otimes(y-y_0)+(x-x_0)\otimes y_0+(x-x_0)\otimes(y-y_0)}_\api\\
   \le\norm{x_0}\norm{y-y_0}+\norm{x-x_0}\norm{y_0}
       +\norm{x-x_0}\norm{y-y_0}
   \le 3\varepsilon.
  \end{multline*}
  Since $x_0\otimes y_0$ is in $\Span\{e_i\otimes f_j\mid
  i,j\in\mathbb N\}$, it follows that $x\otimes y$ can be approximated
  by elements of the span. It follows that the span is dense in
  $E\otimes F$ and, therefore, in $E\ftp F$.
\end{proof}

\subsection{The diagonal of $E\ftp E$}

Suppose that $E$ is a Banach lattice with a basis
$(e_i)$. As we just observed, $(e_i\otimes e_j)_{i,j}$ is a
1-unconditional basis in $E\ftp E$. It is easy to see that
\begin{eqnarray}
  \notag \Ioc&=&\cspan\bigl\{(e_i\otimes e_j)\mid i\ne j\bigr\}\\
  \label{diag-basis}(E\ftp E)/\Ioc&=&\cspan\bigl\{(e_i\otimes e_i)\mid i\in\mathbb N\}
\end{eqnarray}
In particular, we can view $\Ioc$ and $(E\ftp E)/\Ioc$ as two mutually
complementary bands in $E\ftp E$.  In view of this, our interpretation
of $(E\ftp E)/\Ioc$ as the diagonal of $E\ftp E$ is consistent with,
e.g., Examples~2.10 and 2.23 in~\cite{Ryan:02}.

It follows immediately from Corollary~\ref{general-normed} that
$(E\ftp E)/\Ioc$ is lattice isometric to~$\overline{E_{(2)}}$.
Moreover, in view of \eqref{diag-basis}, the map $T$ in
Corollary~\ref{general-normed} has a particularly simple form:
$T\colon e_i\to e_i\otimes e_i$. Thus, Corollary~\ref{general-normed}
for Banach lattices with a basis can be stated as follows.

\begin{theorem}\label{m-basis}
  Suppose that $E$ is a Banach lattice with a
  basis $(e_i)$. Then the map that sends
  $\sum_{i=1}^\infty u_ie_i\otimes e_i$ in $E\ftp E$ into
  $\sum_{i=1}^\infty u_i\odot e_i$ in $\overline{E_{(2)}}$ is a surjective
  lattice isometry between $(E\ftp E)/\Ioc$ and~$\overline{E_{(2)}}$.
\end{theorem}

Combining this with Propositions~\ref{p-conv-BL}
and~\ref{p-est-basis}, we get the following corollaries.

\begin{corollary}
  Suppose that $E$ is a Banach lattice with a basis. If $E$ is
  2-convex then $(E\ftp E)/\Ioc$ is lattice isometric to~$E_{(2)}$.
\end{corollary}

\begin{corollary}\label{p-est-diag}
  Suppose that $E$ is a Banach lattice with a normalized basis
  $(e_i)$, satisfying a lower 2-estimate with constant~$M$. Then
  $(E\ftp E)/\Ioc$ is lattice isomorphic (isometric if $M=1$) to
  $\ell_1$ via $(x_i)\in\ell_1\mapsto\sum_{i=1}^\infty x_ie_i\otimes
  e_i$.
\end{corollary}

\begin{example}
\label{examp:ellp}
  If $E=\ell_p$ for $1\le p<\infty$ then $E^2$ (and, therefore,
  $E_{(2)}$) can be identified as a vector space with
  $\ell_{\frac{p}{2}}$. In the case $p\ge 2$, $E$ is 2-convex and
  hence $(E\ftp E)/\Ioc = E_{[2]} = E_{(2)} = \ell_{\frac{p}{2}}$. In
  the case $1\le p<2$, $E$ satisfies the lower 2-estimate and hence
  $(E\ftp E)/\Ioc=E_{[2]} = \ell_1$. On the other hand, in the
  latter case, $\norm{\cdot}_{(2)}$ is the $\ell_1$-norm on
  $E_{(2)}=\ell_{\frac{p}{2}}$ and we have
  \begin{math}
    E_{[2]}
    =\overline{(E_{(2)},\norm{\cdot}_{(2)})}
    =\overline{(\ell_{\frac{p}{2}},\norm{\cdot}_{\ell_1})} 
    =\ell_1.
  \end{math}
\end{example}


\providecommand{\bysame}{\leavevmode\hbox to3em{\hrulefill}\thinspace}

\end{document}